\documentclass{amsart}
\usepackage{amssymb,latexsym}
\usepackage{enumerate}
\usepackage{color}
\usepackage{mathrsfs}
\usepackage{graphicx}
\usepackage[colorlinks=true,citecolor=blue]{hyperref}

\theoremstyle{theorem}
\newtheorem{theorem}{Theorem}[section]
\newtheorem{proposition}{Proposition}[section]
\newtheorem{corollary}{Corollary}[section]
\newtheorem{lemma}{Lemma}[section]

\theoremstyle{definition}
\newtheorem{definition}{Definition}[section]
\newtheorem{example}{Example}[section]
\theoremstyle{question}
\newtheorem{question}{Question}

\theoremstyle{remark}
\newtheorem{remark}{Remark}

\begin{document}

\title[Sensitivity and transitivity of fuzzified dynamical systems]
{Sensitive dependence and transitivity of fuzzified dynamical systems}

\author[X. Wu]{Xinxing Wu}
\address[X. Wu]{School of Sciences, Southwest Petroleum University, Chengdu, Sichuan, 610500, People's
Republic of China}
\email{wuxinxing5201314@163.com}

\author[X. Wang]{Xiong Wang}
\address[X. Wang]{Institute for Advanced Study, Shenzhen University,Nanshan District Shenzhen, Guangdong, China}
\email{wangxiong8686@szu.edu.cn}
\thanks{Corresponding author: Xiong Wang (wangxiong8686@szu.edu.cn)}

\author[G. Chen]{Guanrong Chen}
\address[G. Chen]{Department of Electronic
Engineering, City University of Hong Kong, Hong Kong SAR, People's
Republic of China} \email{gchen@ee.cityu.edu.hk}
\date{\today}


\maketitle
\begin{abstract}
This paper proves that a set-valued dynamical system is sensitively dependent on initial conditions
(resp., $\mathscr{F}$-sensitive, multi-sensitive)
if and only if its $g$-fuzzification is sensitively dependent on initial conditions
(resp., $\mathscr{F}$-sensitive, multi-sensitive), where $\mathscr{F}$ is a Furstenberg family.
As an application, it is shown that there exists a sensitive dynamical system whose $g$-fuzzification
does not have such sensitive dependence for any $g$ in a certain domain.
Moreover, a sufficient condition ensuring that the $g$-fuzzification
of every nontrivial dynamical system is not transitive is obtained.
These give an answer to a question posed in \cite[J. Kupka, Information Sciences,
{\bf 279} (2014): 642--653]{Kupka2014}.
\end{abstract}

\section{Introduction}

A {\it dynamical system} is a pair $(X, f)$,
where $X$ is a compact metric space with a metric $d$ and
$f: X\longrightarrow X$ is a continuous
map. Let $\mathbb{N}=\{1, 2, 3, \ldots\}$ and
$\mathbb{Z}^{+}=\{0, 1, 2, \ldots\}$. The complexity of a dynamical system has been a central topic of research since the term
of chaos was introduced by Li and Yorke \cite{Li75} in 1975, known as
{\it Li-Yorke chaos} today. An essential feature of chaos is the
impossibility of prediction of its long-term dynamics
due to the exponential separation of any two nearby bounded orbits.

Another interesting question about a dynamical system is when orbits from nearby
points start to deviate after finite steps. This is also one of the most important features
depicting the chaoticity of a system. This  concept has been widely studied and is termed as {\it sensitive
dependence on initial conditions} (briefly, {\it sensitivity}),
 detailed by Auslander and Yorke \cite{AY80} and  further popularized
by Devaney \cite{De89}. More precisely, a dynamical system $(X, f)$ is {\it sensitively dependent}
if there exists $\delta>0$ such that for any $x\in X$
and any $\varepsilon>0$, there exist $y\in B_{d}(x, \varepsilon):=\left\{z\in X: d(x, z)<\varepsilon\right\}$
and $n\in \mathbb{Z}^{+}$ satisfying $d(f^{n}(x), f^{n}(y))>\delta$.

In the rest of this Introduction, some notations are used without precise definitions, which will be given in the following section.
Given a dynamical system $(X, f)$, one can obtain two associated systems induced by $(X, f)$.
One is $(\mathcal{K}(X), \overline{f})$ on the hyperspace $\mathcal{K}(X)$ consisting
of all nonempty closed subsets of $X$ with the Hausdorff metric. The other is
its $g$-fuzzification system $(\mathbb{F}(X), \widetilde{f}_{g})$ on the space $\mathbb{F}(X)$
consisting of all upper semicontinuous fuzzy sets with a levelwise metric.
The notion of $g$-fuzzification was introduced by Kupka in \cite{Kupka2011-1}.

Afterwards, the connection of dynamical properties among $(X, f)$, $(\mathcal{K}(X), \overline{f})$
and $(\mathbb{F}(X), \widetilde{f}_{g})$ has been studied by many researchers.

Bauer and Sigmund \cite{BS1975} studied the interplay of chaos
in dynamical systems (individual chaos) with the corresponding set-valued versions
(collective chaos). Then, Banks \cite{Banks2005} proved that $f$ is weakly mixing
if and only if $\overline{f}$ is transitive, which is equivalent to
the weakly mixing property of $\overline{f}$ (also see \cite{Liao2006}).
Guirao et al. \cite{GKLOP} proved that $\overline{f}$ has the same type of chaos
as $f$ (distributional chaos, Li-Yorke chaos, $\omega$-chaos,
topological chaos, specification
property, exact Devaney chaos, total Devaney chaos). Hou et al. \cite{Hou2008}
showed that if $f$ is a non-minimal $M$-system, then $\overline{f}$ is sensitive. In
\cite{Gu2007}, Gu proved that the sensitivity of $\overline{f}$ implies that $f$ is also sensitive.
Lately, Liu et al. \cite{Liu} gave examples to show that the converse may
not hold, i.e., the sensitivity of $f$ does not necessarily imply the
sensitivity of $\overline{f}$, and they proved that if $f$ is a surjective
continuous interval map then the sensitivities of
$\overline{f}$ and $f$ are equivalent. Li \cite{Li2012} showed that the
multi-sensitivity of $f$ and $\overline{f}$ are equivalent properties.
Recently, we \cite{WWC} studied the sensitivity of $(\mathcal{K}(X), \overline{f})$
in Furstenberg families. In particular, we proved that $\mathscr{F}$-sensitivity
of $(\mathcal{K}(X), \overline{f})$ implies that of $(X, f)$, and the converse is also
true if the Furstenberg family $\mathscr{F}$ is a filter
\cite[Corollary 1, Theorem 4]{WWC}. For more recent results on the notion of sensitivity, one is referred to
\cite{HSZ15,HKZ15,HLY11,Mo2007,WWC,WC16,WOC} and some references therein.

Rom\'{a}n-Flores and  Chalco-Cano \cite{Roman2008} studied some chaotic properties (for example,
transitivity, sensitive dependence, periodic density) for Zadeh's extension
of a dynamical system. In \cite{Kupka2011}, Kupka investigated the relations
between Devaney chaos in the original system and in Zadeh's extension
system. Especially, he proved that Zadeh's extension is periodically dense in $\mathbb{F}(X)$
(resp. $\mathbb{F}^{\lambda}(X)$ for any $\lambda\in (0, 1]$) if and only if $\overline{f}$ is
periodically dense in $\mathcal{K}(X)$. Recently, Kupka \cite{Kupka2011-1} introduced
the notion of $g$-fuzzification which is a generalization of Zadeh's extension and proved
that a dynamical system is continuous if and only if its $g$-fuzzification system
is continuous. In \cite{Kupka2014}, he continued in studying chaotic properties
(for example, Li-Yorke chaos, distributional chaos, $\omega$-chaos, transitivity,
total transitivity, exactness, sensitive dependence, weakly mixing, mildly mixing, topologically mixing)
of $g$-fuzzification systems and showed that if the $g$-fuzzification
$(\mathbb{F}^{1}(X), \widetilde{f}_{g})$ has the property $P$, then $(X, f)$
also has the property $P$, where $P$ denotes the following properties: exactness, sensitive dependence,
weakly mixing, mildly mixing, or topologically mixing. Meanwhile, he posed the following question:

\begin{question}\cite{Kupka2014}\label{Question 1}
Does the $P$-property of $(X, f)$ imply the
$P$-property of $(\mathbb{F}^{1}(X), \widetilde{f}_{g})$?
\end{question}

In this paper, we further investigate the relationships between the sensitivity and the
transitivity of set-valued dynamical systems and $g$-fuzzification
through further developing the results in \cite{Kupka2014}. In this study, we prove
that $\overline{f}$ is sensitively dependent if $(\mathbb{F}_{0}(X), \widetilde{f}_{g})$
is sensitively dependent. Combining this with \cite[Proposition 2.1, Proposition 2.2]{Liu}, we
give a negative answer to Question~\ref{Question 1} above on sensitivity.
Moreover, we obtain the following results:
\begin{enumerate}[(1)]
\item $(\mathcal{K}(X), \overline{f})$ is sensitively dependent $\Longleftrightarrow$
$(\mathbb{F}^{1}(X), \widetilde{f}_{g})$ is sensitively dependent for some $g\in D_{m}(I)$
with $g^{-1}(1)=\{1\}$ $\Longleftrightarrow$ $(\mathbb{F}^{1}(X), \widetilde{f}_{g})$ is sensitively dependent for any $g\in D_{m}(I)$
with $g^{-1}(1)=\{1\}$.
\item There exists $g\in D_{m}(I)$ such that for every nontrivial dynamical system
$(X, f)$, $(\mathbb{F}^{1}(X), \widetilde{f}_{g})$ is not transitive (thus, not weakly mixing).
\end{enumerate}

This paper is organized as follows: in Section \ref{S-2}, some basic definitions and notations are introduced.
In Section \ref{S-3}, some results obtained in \cite{Kupka2011-1,Kupka2014}
are corrected. Then, in Sections \ref{S-4} and \ref{S-5}, some preliminary results on the sensitivity,
$\mathscr{F}$-sensitivity, and multi-sensitivity are established and negatively answer Question \ref{Question 1} on sensitivity.
Finally, the transitivity is studied in Section \ref{S-6}.

\section{Basic definitions and notations}\label{S-2}

\subsection{Furstenberg family, transitivity, and sensitivity}
First, recall some basic concepts related to the Furstenberg families
(see \cite{akin97} for more details).

Let $\mathcal{P}$ be the
collection of all subsets of $\mathbb{Z}^{+}$. A collection
$\mathscr{F}\subset\mathcal{P}$ is called a {\it Furstenberg family}
if it is hereditary upwards, i.e., $F_{1}\subset F_{2}$ and
$F_{1}\in\mathscr{F}$ together imply $F_{2}\in\mathscr{F}$. A family
$\mathscr{F}$ is {\it proper} if it is a proper subset of
$\mathcal{P}$, i.e. neither empty nor the whole $\mathcal{P}$.
Throughout this paper, all Furstenberg families are
proper. It is clear that a family $\mathscr{F}$ is proper if
and only if $\mathbb{Z}^{+}\in\mathscr{F}$ and
$\emptyset\notin\mathscr{F}$. Let $\mathscr{F}_{inf}$ be a
Furstenberg family of all infinite subsets of $\mathbb{Z}^+$. For a family $\mathscr{F}$, its {\it dual
family} is
\[
\kappa\mathscr{F}=\left\{F\in \mathcal{P}:
\mathbb{Z}^{+}\setminus F\notin\mathscr{F}\right\}.
\]
It is easy to verify that $\kappa\mathscr{F}$ is a Furstenberg family,
and is proper if $\mathscr{F}$ is so. For Furstenberg families $\mathscr{F}_1$
and $\mathscr{F}_2$, let $\mathscr{F}_{1}\cdot \mathscr{F}_{2}=\left\{F_{1}\cap F_{2}:
F_{1}\in \mathscr{F}_{1}, F_{2}\in \mathscr{F}_{2}\right\}$.
A Furstenberg family $\mathscr{F}$ is a {\it filter} if $\mathscr{F}$ is
proper and $\mathscr{F}\cdot\mathscr{F}\subset \mathscr{F}$.

For $U, V\subset X$, define the {\it return time set from $U$
to $V$} as $ N(U, V)=\{n\in \mathbb{Z}^{+}:
f^{n}(U)\cap V\neq \emptyset\}$. In particular,
$N(x, V)=\left\{n\in \mathbb{Z}^{+}:
f^{n}(x)\in V\right\}$ for $x\in X$.

A dynamical system $(X, f)$ is
\begin{enumerate}[(1)]
\item {\it transitive} if for every pair of nonempty open subsets $U, V$ of $X$, $N(U, V)\neq \emptyset$;
\item {\it (topologically) weakly mixing}
if $(X \times X, f\times f)$ is transitive;
\item {\it mixing} if for every pair of nonempty open subsets $U, V$ of $X$, there exists
$N\in \mathbb{N}$ such that $[N, +\infty)\subset
N(U, V)$.
\end{enumerate}

The ``largeness'' of the time set where sensitivity emerges can be regarded as
a measure of how sensitive a system is. For this reason, Moothathu \cite{Mo2007}
proposed three stronger forms of sensitivity: syndetic sensitivity,
cofinite sensitivity (also called strong sensitivity in \cite{Abraham}),
and multi-sensitivity. Then, Tan and Zhang \cite{Tan-Zhang}
introduced a more general description of sensitivity by using Furstenberg families.

\begin{definition}\cite{Li2012, Mo2007, Tan-Zhang}\label{D-1}
Let $(X, f)$ be a system and $\mathscr{F}$ be a Furstenberg family.
\begin{enumerate}[(1)]
\item $(X, f)$ is {\it multi-sensitive} if there exists $\varepsilon>0$
(multi-sensitive constant) such
that for any $k\in \mathbb{N}$ and nonempty open subsets ${U}_{1}, \ldots, {U}_{k}
\subset X$, $\bigcap_{i=1}^{k}\{n\in \mathbb{Z}^{+}: \mathrm{diam}(f^{n}({U}_{i}))
>\varepsilon\} \neq\emptyset$, i.e., there exists $n\in \mathbb{Z}^{+}$ such that
$\mathrm{diam}(f^{n}({U}_{i})) >\varepsilon$ holds for all $i=1, \ldots, k$,
where $\mathrm{diam}(\cdot)$ denotes the diameter of a given set.

\item $(X, f)$ is {\it $\mathscr{F}$-sensitive} if there exists $\varepsilon>0$
($\mathscr{F}$-sensitive constant) such that for any
nonempty open subset ${U}\subset X$, $\{n\in \mathbb{Z}^{+}:
\mathrm{diam}(f^{n}({U}))>\varepsilon\}\in \mathscr{F}$.

%
\end{enumerate}
\end{definition}

\subsection{Set-valued dynamical system}
Let $\mathcal{K}(X)$ be the hyperspace on $X$, i.e., the space of nonempty
compact subsets of $X$ with the Hausdorff metric $d_{H}$ defined by
\[
d_{H}(A, B)=\max\left\{\max_{x\in A}\min_{y\in B}d(x, y), \max_{y\in B}\min_{x\in A}d(x, y)\right\}
\]
for any $A, B\in \mathcal{K}(X)$. Clearly, $(\mathcal{K}(X), d_{H})$ is a compact
metric space. The system $(X, f)$ induces a set-valued dynamical system $(\mathcal{K}(X), \overline{f})$,
where $\overline{f}: \mathcal{K}(X)\longrightarrow \mathcal{K}(X)$ is defined as
$\overline{f}(A)=f(A)$ for any $A\in \mathcal{K}(X)$.
For any finite collection $A_{1}, \ldots, A_{n}$ of nonempty subsets of $X$, let
\[
\langle A_{1}, \ldots, A_{n}\rangle=\left\{A\in \mathcal{K}(X): A\subset \bigcup_{i=1}^{n}A_{i},
\ A\cap A_{i}\neq \emptyset \text{ for all } i=1, \ldots, n\right\}.
\]
It follows from \cite{Illanes1999} that the topology on $\mathcal{K}(X)$
given by the metric $d_{H}$ is same as the Vietoris or finite
topology, which is generated by a basis consisting of all sets of the following form:
\[
\begin{split}
\langle {U}_{1}, \ldots, {U}_{n}\rangle, &
\text{ where } {U}_{1}, \ldots, {U}_{n} \text{ are an arbitrary finite collection }\\
& \text {of nonempty open subsets of
 } X.
 \end{split}
\]

\subsection{$g$-fuzzification}

{\it A fuzzy set $A$} in space $X$ is a function $A: X\longrightarrow I$,
where $I= [0, 1]$. Given a fuzzy set $A$,
its {\it $\alpha$-cuts} (or {\it $\alpha$-level sets}) $[A]_{\alpha}$
and {\it support} $\mathrm{supp}(A)$ are defined respectively by
\[
[A]_{\alpha}=\{x\in X: A(x)\geq \alpha\}, \quad \forall \alpha\in I,
\]
and
\[
\mathrm{supp}(A)=\overline{\left\{x\in X: A(x)>0\right\}}.
\]

Let $\mathbb{F}(X)$ denote the set of all upper semicontinuous fuzzy sets defined on $X$
and set
\[
\mathbb{F}^{\lambda}(X)=\left\{A\in \mathbb{F}(X): A(x)\geq \lambda \text{ for some }
x\in X\right\}.
\]

Define $\emptyset_{X}$ as the empty fuzzy set ($\emptyset_{X}\equiv 0$) in $X$,
and $\mathbb{F}_{0}(X)$ as the set of all nonempty upper semicontinuous fuzzy sets.
Since the Hausdorff metric $d_{H}$ is measured only between two nonempty closed subsets
in $X$, one can consider the following extension of the Hausdorff metric:
\[
d_{H}(\emptyset, \emptyset)=0, \text{ and } d_{H}(\emptyset, A)=d_{H}(A, \emptyset)=
\mathrm{diam}X, \ \forall A\in \mathcal{K}(X).
\]
Using this Hausdorff metric, one can define a {\it levelwise metric} $d_{\infty}$ on $\mathbb{F}(X)$ by
\[
d_{\infty}(A, B)=
\sup\left\{d_{H}([A]_{\alpha}, [B]_{\alpha}): \alpha\in (0, 1]\right\}, \ \forall A, B\in \mathbb{F}(X).
\]
It is well known that the spaces $(\mathbb{F}(X), d_{\infty})$
and $(\mathbb{F}^1(X), d_{\infty})$ are complete, but not compact and not
separable (see \cite{Kupka2011-1} and references therein).

A fuzzy set $A\in \mathbb{F}(X)$ is {\it piecewise constant}, if there exists a finite
number of sets $D_{i}\subset X$ such that $\bigcup \overline{D_{i}}=X$ and $A|_{\mathrm{int}\overline{D_{i}}}$
is constant. $A$ can be represented by a sequence of closed subsets $\{A_{1}, A_{2}, \ldots, A_{k}\}\subset X$
and an increasing sequence of reals $\{\alpha_{1}, \alpha_{2}, \ldots, \alpha_{k}\}\subset (0, 1]$ if
\[
[A]_{\alpha}=A_{i+1}, \text{ whenever } \alpha\in (\alpha_{i}, \alpha_{i+1}].
\]

Kupka \cite{Kupka2011} proved that the set of all piecewise constant upper
continuous maps is dense in $\mathbb{F}(X)$. Then, in \cite{Kupka2011-1}
he introduced the notion of $g$-fuzzification
to generalize Zadeh's extension.

{\it Zadeh's extension} of a dynamical system
$(X, f)$ is a map $\widetilde{f}: \mathbb{F}(X)\longrightarrow \mathbb{F}(X)$
defined by
\[
\widetilde{f}(A)(x)=\sup\left\{A(y): y\in f^{-1}(x)\right\},
\text{ for any } A\in \mathbb{F}(X) \text{ and any }
x\in X.
\]

Denote by $D_{m}(I)$ the set of all nondecreasing right-continuous functions
$g: I\longrightarrow I$ with $g(0)=0$ and $g(1)=1$. For a dynamical system $(X, f)$
and for any $g\in D_{m}(I)$, define a map $\widetilde{f}_{g}: \mathbb{F}(X)
\longrightarrow \mathbb{F}(X)$ by
\[
\widetilde{f}_{g}(A)(x)=\sup\left\{g(A(y)): y\in f^{-1}(x)\right\},
\text{ for any } A\in \mathbb{F}(X) \text{ and any }
x\in X,
\]
which is called the {\it $g$-fuzzification} of the dynamical system $(X, f)$.
Clearly, $\widetilde{f}=\widetilde{f}_{\mathrm{id}_{I}}$.

Also, define the {\it $\alpha$-cut} $[A]_{\alpha}^{g}$ of a fuzzy set
$A\in \mathbb{F}(X)$ with respect to $g\in D_{m}(I)$ by
\[
[A]_{\alpha}^{g}=\left\{x\in \mathrm{supp}(A): g(A(x))\geq \alpha\right\}.
\]

In \cite{Kupka2011-1, Kupka2014}, Kupka claimed the following:
\begin{lemma}\cite[Lemma 3]{Kupka2011-1}\label{Lemma 2.1}
Let $(X, f)$ be a dynamical system, and $g\in D_{m}(I)$. Then, for any $A\in \mathbb{F}_{0}(X)$
and any $\alpha\in (0, 1]$, $f([A]_{\alpha}^{g})=\left[\widetilde{f}_{g}(A)\right]_{\alpha}$.
\end{lemma}
\begin{lemma}\cite[Lemma 5]{Kupka2011-1}\label{Lemma 2.2}
Let $g\in D_{m}(I)$, $A\in \mathbb{F}_{0}(X)$,
and $\alpha\in (0, 1]$. If $[A]_{\alpha}^{g}\neq \emptyset$, then
there exists $c\in (0, 1]$ such that $[A]_{\alpha}^{g}=[A]_{c}$.
\end{lemma}
\begin{lemma}\cite[Lemma 6]{Kupka2014}\label{Lemma 2.3}
Let $(X, f)$ be a dynamical system, and $g\in D_{m}(I)$. Then, for any $A\in \mathbb{F}(X)$
and any $\alpha\in [0, 1]$, $f([A]_{\alpha})=\left[\widetilde{f}_{g}(A)\right]_{g(\alpha)}$.
\end{lemma}

\section{Some remarks and new results}\label{S-3}
Firstly, we use an example to show that the proof of Lemma~\ref{Lemma 2.2} given in \cite{Kupka2011}
and the above-stated Lemma \ref{Lemma 2.3} do not hold.
\begin{example}
Let $X=[0, 1]$ and $f: X\longrightarrow X$ defined by $f(x)=x$ for all $x\in [0, 1]$. Define
$g: [0, 1]\longrightarrow [0, 1]$ by
\[
g(x)=\left\{\begin{array}{cc}
2x, & {\rm } \ x\in [0, \frac{1}{2}], \\\
1, & {\rm } \ x\in [\frac{1}{2}, 1],
\end{array}
\right.
\]
and take $A\in \mathbb{F}(X)$ with $A=f$. It is easy to see that
\[
f([A]_{1})=f(\{1\})=\{1\},
\]
and
\[
\left[\widetilde{f}_{g}(A)\right]_{g(1)}=\left\{x\in [0, 1]: \widetilde{f}_{g}(A)(x)
\geq 1\right\}=\left\{x\in [0, 1]: g(A(x))\geq 1\right\}=[1/2, 1].
\]
So, the above-stated Lemma \ref{Lemma 2.3}
does not hold. In the proof of Lemma \ref{Lemma 2.2} given in [9], it was claimed that for any $c\in (0, 1]$
with $g(c)>0$, $[A]_{g(c)}^{g}=[A]_{c}$. However, let us simply choose $c=1$. It can be verified
that $[A]_{g(1)}^{g}=\left\{x\in [0, 1]: g(A(x))\geq 1\right\}=[1/ 2, 1]\neq
\{1\}=[A]_1$.
\end{example}

For any $g\in D_{m}(I)$, the right-continuity of $g$
implies that $\min g^{-1}([x, 1])$ exists for any $x\in [0, 1]$.
Since $g$ is nondecreasing, $\min g^{-1}([x, 1])>0$
holds for any $x\in (0, 1]$. Define $\xi_{g}: [0, 1]\longrightarrow [0,1]$
by $\xi_{g}(x)=\min g^{-1}([x, 1])$ for any $x\in [0, 1]$. Clearly, $\xi_{g}$
is nondecreasing.

Next, we give a correct statement and proof Lemma \ref{Lemma 2.2}.

\begin{lemma}\label{lemma *}
Let $g\in D_{m}(I)$, $A\in \mathbb{F}_{0}(X)$,
and $\alpha\in (0, 1]$. Then,
there exists $c\in (0, 1]$ such that $[A]_{\alpha}^{g}=[A]_{\xi_{g}(\alpha)}$.
\end{lemma}
\begin{proof}
Because $g$ is nondecreasing, it follows that
\[
[A]_{\alpha}^{g}=\left\{x\in \mathrm{supp}(A): g(A(x))\geq \alpha\right\}
=\left\{x\in \mathrm{supp}(A): A(x)\geq \xi_{g}(\alpha)\right\}=[A]_{\xi_{g}(\alpha)}.
\]
\end{proof}
\begin{proposition}\label{Proposition 3.1}
Let $(X, f)$ be a dynamical system, $g\in D_{m}(I)$, and $\widetilde{f}_{g}$
be the $g$-fuzzification of $f$. Then, for any $n\in \mathbb{N}$, any $A\in \mathbb{F}(X)$,
and any $\alpha\in (0, 1]$, $\left[(\widetilde{f}_{g})^{n}(A)\right]_{\alpha}
=f^{n}([A]_{\xi_{g}^{n}(\alpha)})$. In particular, for any $B\in \mathcal{K}(X)$,
$\left[(\widetilde{f}_{g})^{n}(\chi_{B})\right]_{\alpha}
=f^{n}(B)$.
\end{proposition}
\begin{proof}
Applying Lemma \ref{Lemma 2.1} and Lemma \ref{lemma *}, it follows that
$\left[(\widetilde{f}_{g})(A)\right]_{\alpha}=f([A]_{\alpha}^{g})
=f^{n}([A]_{\xi_{g}(\alpha)})$. Applying mathematical induction, it is not difficult
to verify that the proposition is true.
\end{proof}
\begin{proposition}
Let $(X, f)$ be a dynamical system, and $g\in D_{m}(I)$. Then, for any $A\in \mathbb{F}(X)$
and any $\alpha\in [0, 1]$, $f([A]_{\alpha})\subset \left[\widetilde{f}_{g}(A)\right]_{g(\alpha)}$.
\end{proposition}
\begin{proof}
For any $x\in f([A]_{\alpha})$, there exists $y\in [A]_{\alpha}$ such that $x=f(y)$. It is easy
to see that $\widetilde{f}_{g}(A)(x)=g(A(y))\geq g(\alpha)$. So, $x\in \left[\widetilde{f}_{g}(A)\right]_{g(\alpha)}$.
\end{proof}
\begin{lemma}\label{Lemma **}
Let $(X, f)$ be a dynamical system, and $g\in D_{m}(I)$ with $g^{-1}(1)=\{1\}$.
Then, for any $n\in \mathbb{N}$ and any $A\in \mathbb{F}(X)$,
$\left[(\widetilde{f}_{g})^{n}(A)\right]_{1}
=f^{n}([A]_{1})$.
\end{lemma}
\begin{proof}
Applying Proposition \ref{Proposition 3.1}, noting that $g^{-1}(1)=\{1\}$,
it is easy to verify that the lemma is true.
\end{proof}

\section{Sensitivity of $g$-fuzzification}\label{S-4}

This section is devoted to studying the sensitivity of the
$g$-fuzzification systems.

\begin{proposition}\label{Lemma 1}
Let $(X, f)$ be a dynamical system, $A\in \mathcal{K}(X)$ and $\{B_{\lambda}\}_{\lambda\in \Lambda}
\subset \mathcal{K}(X)$. If there exists $\xi>0$ such that for any $\lambda\in \Lambda$, $d_{H}(A, B_{\lambda})<\xi$,
then $d_{H}(A, \overline{\cup_{\lambda\in \Lambda}B_{\lambda}})\leq \xi$.
\end{proposition}
\begin{proof}
It is easy to see that for any $\lambda_{0}\in \Lambda$,
\begin{equation}\label{e-1}
\begin{split}
\sup_{x\in A}\inf_{y\in \overline{\cup_{\lambda\in \Lambda}B_{\lambda}}}d(x, y)
& \ \leq \ \sup_{x\in A}\inf_{y\in \cup_{\lambda\in \Lambda}B_{\lambda}}d(x, y)\\
& \ \leq \ \sup_{x\in A}\inf_{y\in B_{\lambda_{0}}}d(x, y)\leq d_{H}(A, B_{\lambda_{0}})
<\xi.
\end{split}
\end{equation}
Now, for any $y\in \overline{\bigcup_{\lambda\in \Lambda}B_{\lambda}}$, consider the
following two cases:
\begin{enumerate}[(a)]
\item if $y\in \bigcup_{\lambda\in \Lambda}B_{\lambda}$, then there exists $\lambda\in \Lambda$
such that $y\in B_{\lambda}$. Thus, $\inf_{x\in A}d(x, y)\leq d_{H}(A, B_{\lambda})<\xi$;
\item if $y\in \overline{\bigcup_{\lambda\in \Lambda}B_{\lambda}}\setminus \bigcup_{\lambda\in \Lambda}B_{\lambda}$,
then for any $n\in \mathbb{N}$, there exist $\lambda_{n}\in \Lambda$ and $z\in B_{\lambda_{n}}$ such that
$d(y, z)<1/n$. According to the definition of $d_{H}(A, B_{\lambda_{n}})$, there exists
$x_{n}\in A$ such that $d(x_{n}, z)\leq d_{H}(A, B_{\lambda_{n}})<\xi$. So,
\[
\inf_{x\in A}d(x, y)
\leq d(x_{n}, y)\leq d(x_{n}, z)+ d(z, y)<\xi+ \frac{1}{n}.
\]
\end{enumerate}
This implies that $\sup_{y\in \overline{\bigcup_{\lambda\in \Lambda}B_{\lambda}}}\inf_{x\in A}
d(x, y)\leq \xi$. Combining this with \eqref{e-1}, it follows that
$d_{H}(A, \overline{\bigcup_{\lambda\in \Lambda}B_{\lambda}})\leq \xi$.
\end{proof}
\begin{theorem}\label{Theorem 3.1}
Let $(X, f)$ be a dynamical system and $g\in D_{m}(I)$. If
$(\mathbb{F}_{0}(X), \widetilde{f}_{g})$ is sensitively dependent, then
$(\mathcal{K}(X), \overline{f})$ is sensitively dependent.
\end{theorem}
\begin{proof}
Let $\varepsilon>0$ be a sensitive constant of $\widetilde{f}_{g}$.
Given any fixed $A\in \mathcal{K}(X)$ and any $\delta>0$, noting that $\chi_{A}\in \mathbb{F}_{0}(X)$,
the sensitivity of $\widetilde{f}_{g}$ implies that there exist $B\in \mathbb{F}_{0}(X)$ with
$d_{\infty}(\chi_{A}, B)<\frac{\delta}{2}$ and $n\in \mathbb{Z}^{+}$ such that
\begin{equation}\label{e-2}
d_{\infty}((\widetilde{f}_{g})^{n}(\chi_{A}), (\widetilde{f}_{g})^{n}(B))>\varepsilon.
\end{equation}
Applying Proposition \ref{Proposition 3.1}, it follows that for any $\alpha\in (0, 1]$, there exists
$\xi_{g}^{n}(\alpha)\in (0, 1]$ such that
\[
\left[(\widetilde{f}_{g})^{n}(\chi_{A})\right]_{\alpha}=\overline{f}^{n}(A), \text{ and }
\left[(\widetilde{f}_{g})^{n}(B)\right]_{\alpha}=\overline{f}^{n}([B]_{\xi_{g}^{n}(\alpha)}).
\]
This, together with \eqref{e-2}, implies that there exists $\alpha_{0}\in (0, 1]$ such that
\[
d_{H}\left(\left[(\widetilde{f}_{g})^{n}(\chi_{A})\right]_{\alpha_{0}}, \left[(\widetilde{f}_{g})^{n}(B)\right]_{\alpha_{0}}\right)
=d_{H}\left(\overline{f}^{n}(A), \overline{f}^{n}([B]_{\xi_{g}^{n}(\alpha_0)})\right)>\varepsilon.
\]
Clearly,
\[
d_{H}(A, [B]_{\xi_{g}^{n}(\alpha_0)})=d_{H}([\chi_A]_{\xi_{g}^{n}(\alpha_0)},
[B]_{\xi_{g}^{n}(\alpha_0)})\leq d_{\infty}(\chi_{A}, B)<\delta.
\]
Since $A$ and $\delta$ are arbitrary, it is concluded that $\overline{f}$ is sensitively dependent.
\end{proof}
\begin{theorem}\label{Theorem 1}
Let $(X, f)$ be a dynamical system. Then, the following statements are equivalent:
\begin{enumerate}[(1)]
\item\label{4.2.1} $(\mathcal{K}(X), \overline{f})$ is sensitively dependent;
\item\label{4.2.2} $(\mathbb{F}^{1}(X), \widetilde{f}_{g})$ is sensitively dependent for some $g\in D_{m}(I)$
with $g^{-1}(1)=\{1\}$;
\item\label{4.2.3} $(\mathbb{F}^{1}(X), \widetilde{f}_{g})$ is sensitively dependent for any $g\in D_{m}(I)$
with $g^{-1}(1)=\{1\}$.
\end{enumerate}
\end{theorem}
\begin{proof}
(\ref{4.2.2})$\Longrightarrow$(\ref{4.2.1}). This follows immediately from Theorem \ref{Theorem 3.1}.

(\ref{4.2.1})$\Longrightarrow$(\ref{4.2.3}). Let $\varepsilon>0$ be a sensitive constant of $\overline{f}$
and fix any $g\in D_{m}(I)$. For any $A\in \mathbb{F}^{1}(X)$ and any $\delta>0$,
clearly, $[A]_{1}\in \mathcal{K}(X)$. Since $\overline{f}$ is sensitive, there exist $C\in \mathcal{K}(X)$
with $d_{H}([A]_{1}, C)<\frac{\delta}{4}$ and $n\in \mathbb{Z}^{+}$ such that $d_{H}(\overline{f}^{n}([A]_{1}),
\overline{f}^{n}(C))>\varepsilon$. The continuity of $\overline{f}$ implies that there exists
$0<\xi<\delta/4$ such that for any $F\in \left\{F_{1}\in \mathcal{K}(X): d_{H}(F_{1}, C)\leq \xi\right\}
$, $d_{H}(\overline{f}^{n}(C), \overline{f}^{n}(F))<
\frac{d_{H}(\overline{f}^{n}([A]_{1}), \overline{f}^{n}(C))-\varepsilon}{2}$. Set
$Q=\overline{\left\{y\in X: \inf_{x\in C}d(x, y)<\xi\right\}}\in \mathcal{K}(X)$.
Clearly, $d_{H}(Q, C)\leq \xi$. So,
\begin{equation}\label{e-3}
d_{H}(\overline{f}^{n}([A]_{1}), \overline{f}^{n}(Q))
\geq d_{H}(\overline{f}^{n}([A]_{1}), \overline{f}^{n}(C))-d_{H}(\overline{f}^{n}(C), \overline{f}^{n}(Q))
>\varepsilon.
\end{equation}

Take $X_{1}=X$, $\alpha_{1}=\max_{x\in X_{1}}A(x)=1$ and $D_{1}=A^{-1}(\alpha_{1})\cap X_{1}=[A]_{1}$. Define inductively
$X_{i+1}=X_{i}\setminus \left\{y\in X: \inf_{x\in D_{i}}d(x, y)<\frac{\delta}{4}\right\}$, $\alpha_{i+1}=
\max_{x\in X_{i+1}}A(x)$ and $D_{i+1}=A^{-1}(\alpha_{i+1})\cap X_{i+1}$ for $i\geq 1$. The compactness
of $X$ implies that there exists $k\in \mathbb{N}$ such that $X_{k+1}=\emptyset$.
Denote ${U}_{i}=\overline{\bigcup_{j=1}^{i}\left\{y\in X: \inf_{x\in D_{j}}d(x, y)<\frac{\delta}{4}\right\}}$
for $i=1, \ldots, k$ and take a piecewise constant fuzzy set $\mathscr{A}$ satisfying $[\mathscr{A}]_{\alpha}={U}_{i}$
for $\alpha\in (\alpha_{i+1}, \alpha_{i}]$. It follows from the
proof of \cite[Lemma 1]{Kupka2011} that $d_{\infty}(A, \mathscr{A})<\frac{\delta}{4}$.

Next, take another piecewise constant fuzzy set $E\in \mathbb{F}^{1}(X)$ such
that
\[
[E]_{\alpha}=\left\{\begin{array}{cc}
Q, & {\rm } \ \alpha\in (\frac{\alpha_{2}+\alpha_{1}}{2}, \alpha_{1}], \\\
{U}_{1}\cup Q, & {\rm } \ \alpha\in (\alpha_{2}, \frac{\alpha_{2}+\alpha_{1}}{2}], \\\
{U}_{i}\cup Q, & {\rm } \ \alpha\in (\alpha_{i+1}, \alpha_{i}], i\in \{2, \ldots, k\}.
\end{array}
\right.
\]
It can be verified that
\[
d_{H}(Q, {U}_{1})\leq d_{H}(Q, C)+d_{H}(C, [A]_{1})+d_{H}([A]_{1}, {U}_{1})\leq \xi
+\frac{\delta}{4}+\frac{\delta}{4}<\frac{3\delta}{4}.
\]
From this, it follows that $d_{\infty}(\mathscr{A}, E)<\frac{3\delta}{4}$,
so that
\[
d_{\infty}(A, E)\leq d_{\infty}(A, \mathscr{A})+d_{\infty}(\mathscr{A}, E)<\delta.
\]
Now, applying Lemma \ref{Lemma **} and \eqref{e-3}, one has
\begin{eqnarray*}
d_{\infty}((\widetilde{f}_{g})^{n}(A), (\widetilde{f}_{g})^{n}(E))
& \geq & d_{H}(\left[(\widetilde{f}_{g})^{n}(A)\right]_{1}, \left[(\widetilde{f}_{g})^{n}(E)\right]_{1})\\
& = & d_{H}\left(\overline{f}^{n}([A]_{1}), \overline{f}^{n}([E]_{1})\right)=
d_{H}\left(\overline{f}^{n}([A]_{1}), \overline{f}^{n}(Q)\right)>\varepsilon.
\end{eqnarray*}

(\ref{4.2.3})$\Longrightarrow$(\ref{4.2.2}). It is obvious.
\end{proof}

Theorem \ref{Theorem 1}, together with the fact that $\widetilde{f}=\widetilde{f}_{\mathrm{id}_{I}}$,
yields the following result.

\begin{corollary}
Let $(X, f)$ be a dynamical system. Then $(\mathcal{K}(X), \overline{f})$ is sensitively
dependent if and only if $(\mathbb{F}^{1}(X), \widetilde{f})$ is sensitively dependent.
\end{corollary}


To close this section, we apply Theorem \ref{Theorem 3.1} and \cite[Proposition 2.1, Proposition 2.2]{Liu}
to construct a counterexample which gives a negative
answer to Question~\ref{Question 1} above.

\begin{example}\label{e-1-1}
Let $\mathbb{R}/\mathbb{Z}$ be the domain of the unite circle $\mathbb{S}^{1}$. Define
a metric $d$ by $d(a, b)=\min\{|a-b|, 1-|a-b|\}$. Then, the
rigid rotation $R_{\alpha}: \mathbb{S}^{1}\longrightarrow \mathbb{S}^{1}$
by a real number $\alpha$ is given by
\[
R_{\alpha}(t)=t+\alpha \ (\mathrm{mod}\ 1), \text{ for all } t\in \mathbb{R}.
\]
Corresponding to the irrational $\alpha$, the Denjoy homeomorphism $D_{\alpha}:
\mathbb{S}^{1}\longrightarrow \mathbb{S}^{1}$ is an orientation preserving homeomorphism
of the circle characterized by the following properties:
\begin{enumerate}
\item the rotational number of $D_{\alpha}$ is $\alpha$;
\item there is a Cantor set $C_{\alpha}\subset \mathbb{S}^{1}$ such that $D_{\alpha}|_{C_{\alpha}}$
is minimal.
\end{enumerate}

In \cite[Proposition 2.1, Proposition 2.2]{Liu}, Liu et al. proved that
$(C_{\alpha}, D_{\alpha}|_{C_{\alpha}})$ is sensitively dependent but
its set-valued dynamical system $(\mathcal{K}(C_{\alpha}), \overline{D_{\alpha}|_{C_{\alpha}}})$
is not sensitively dependent. This, together with Theorem~\ref{Theorem 3.1}, implies
that for every $g\in D_{m}(I)$, the $g$-fuzzification of $(C_{\alpha}, D_{\alpha}|_{C_{\alpha}})$
is not sensitively dependent. This shows that the answer to Question~\ref{Question 1} is negative.
\end{example}

\section{$\mathscr{F}$-sensitivity and multi-sensitivity of $g$-fuzzification}\label{S-5}

As an extension of the last section, this section is devoted to studying $\mathscr{F}$-sensitivity
and multi-sensitivity of $g$-fuzzification.

\begin{theorem}\label{F-sensitivity}
Let $(X, f)$ be a system and $\mathscr{F}$ be a Furstenberg family. Then,
the following statements are equivalent:
\begin{enumerate}[(1)]
\item\label{5.1.1} $(\mathcal{K}(X), \overline{f})$ is $\mathscr{F}$-sensitive;
\item\label{5.1.2} $(\mathbb{F}^{1}(X), \widetilde{f}_{g})$ is $\mathscr{F}$-sensitive for some $g\in D_{m}(I)$
with $g^{-1}(1)=\{1\}$;
\item\label{5.1.3} $(\mathbb{F}^{1}(X), \widetilde{f}_{g})$ is $\mathscr{F}$-sensitive for any $g\in D_{m}(I)$
with $g^{-1}(1)=\{1\}$.
\end{enumerate}
\end{theorem}

\begin{proof}
Similarly to the proof of Theorem \ref{Theorem 3.1}, it can be verified that
(\ref{5.1.3})$\Longrightarrow$ (\ref{5.1.2}) $\Longrightarrow$
(\ref{5.1.2}). It suffices to check that (\ref{5.1.1}) $\Longrightarrow$ (\ref{5.1.3}).

Let $\varepsilon>0$ be a $\mathscr{F}$-sensitive constant of $\overline{f}$
and fix any $g\in D_{m}(I)$. For any $A\in \mathbb{F}^{1}(X)$ and any $\delta>0$,
the $\mathscr{F}$-sensitivity of $\overline{f}$ implies that there exists
$F\in \mathscr{F}$ such that for any $n\in F$, there exists $C\in \mathcal{K}(X)$
with $d_{H}([A]_{1}, C)<\frac{\delta}{4}$ satisfying $d_{H}(\overline{f}^{n}([A]_{1}),
\overline{f}^{n}(C))>\frac{\varepsilon}{2}$. Similarly to the proof of Theorem \ref{Theorem 1},
it follows that there exists $E\in \mathbb{F}^{1}(X)$ such that $d_{\infty}(A, E)<\delta$
and $d_{\infty}((\widetilde{f}_{g})^{n}(A), (\widetilde{f}_{g})^{n}(E))>\varepsilon/ 2$.
This implies that $F\subset \left\{n\in \mathbb{Z}^{+}:
\mathrm{diam}((\widetilde{f}_{g})^{n}(B_{d_{\infty}}(A, \delta)))>\varepsilon/ 2\right\}\in \mathscr{F}$.
So, $(\mathbb{F}^{1}(X), \widetilde{f}_{g})$ is  $\mathscr{F}$-sensitive,
as $A$ and $\delta$ are arbitrary.
\end{proof}

Combining Theorem \ref{F-sensitivity} with \cite[Theorem 4]{WWC}, one can immediately obtain the following.

\begin{corollary}
Let $(X, f)$ be a system and $\mathscr{F}$ be a filter. Then,
the following statements are equivalent:
\begin{enumerate}[(1)]
\item $(X, f)$ is $\mathscr{F}$-sensitive;

\item $(\mathcal{K}(X), \overline{f})$ is $\mathscr{F}$-sensitive;

\item $(\mathbb{F}^{1}(X), \widetilde{f}_{g})$ is $\mathscr{F}$-sensitive for some $g\in D_{m}(I)$
with $g^{-1}(1)=\{1\}$;

\item $(\mathbb{F}^{1}(X), \widetilde{f}_{g})$ is $\mathscr{F}$-sensitive for any $g\in D_{m}(I)$
with $g^{-1}(1)=\{1\}$.
\end{enumerate}
\end{corollary}

Slightly modifying the proofs of Theorem~ \ref{Theorem 1} and Theorem~\ref{F-sensitivity}
and applying \cite[Theorem 3.2, Theorem 3.3]{Li2012},
one can prove the following.

\begin{theorem}\label{multi-sensitivity}
Let $(X, f)$ be a system. Then,
the following statements are equivalent:
\begin{enumerate}[(1)]
\item $(X, f)$ is multi-sensitive;

\item $(\mathcal{K}(X), \overline{f})$ is multi-sensitive;

\item $(\mathbb{F}^{1}(X), \widetilde{f}_{g})$ is multi-sensitive for some $g\in D_{m}(I)$
with $g^{-1}(1)=\{1\}$;

\item $(\mathbb{F}^{1}(X), \widetilde{f}_{g})$ is multi-sensitive for any $g\in D_{m}(I)$
with $g^{-1}(1)=\{1\}$.
\end{enumerate}
\end{theorem}

\section{Transitivity of $g$-fuzzification}\label{S-6}

For the weakly mixing property of $g$-fuzzification, we have the following result:
\begin{theorem}\label{Theorem **}
Let $(X, f)$ be a dynamical system, and $g\in D_{m}(I)$. If
$(\mathbb{F}^{1}(X), \widetilde{f}_{g})$ is transitive, then
$(\mathcal{K}(X), \overline{f})$ is weakly mixing.
\end{theorem}
\begin{proof}
Applying \cite[Theorem 2]{Banks2005}, it suffices to prove that $\overline{f}$
is transitive.

For any pair of nonempty open subsets ${U}, {V}\subset \mathcal{K}(X)$,
there exist $A\in {U}$, $B\in {V}$ and $\delta>0$ such that
$B_{d_{H}}(A, \delta)\subset {U}$ and
$B_{d_{H}}(B, \delta)\subset {V}$. Noting that $B_{d_{\infty}}(\chi_{A}, \delta)$
and $B_{d_{\infty}}(\chi_{B}, \delta)$ are nonempty subsets of $\mathbb{F}^{1}(X)$,
since $\widetilde{f}_{g}$
is transitive, there exists $n\in \mathbb{Z}^{+}$ such that
$(\widetilde{f}_{g})^{n}(B_{d_{\infty}}(\chi_A, \delta))\bigcap B_{d_{\infty}}(\chi_B,
\delta)\neq \emptyset$.
Then, there exists a point $F_{1}\in B_{d_{\infty}}(\chi_A, \delta)$ such that $(\widetilde{f}_{g})^{n}
(F_{1})\in B_{d_{\infty}}(\chi_B, \delta)$. This implies that, for any
$\alpha\in (0, 1]$,
\begin{equation}\label{1}
d_{H}(\left[(\widetilde{f}_{g})^{n}(F_{1})\right]_{\alpha}, B)<\delta.
\end{equation}
In particular, applying Proposition \ref{Proposition 3.1}, it follows that
there exists $\xi\in (0, 1]$ such that
\begin{equation}\label{2}
\left[(\widetilde{f}_{g})^{n}(F_{1})\right]_{1/2}
=\overline{f}^{n}([F_{1}]_{\xi}).
\end{equation}
Since $F_{1}\in B_{d_{\infty}}(\chi_A, \delta)$,
it is easy to see that $d_{H}(A, [F_{1}]_{\xi})<\delta$, i.e., $[F_{1}]_{\xi}\in B_{d_{H}}
(A, \delta)\subset {U}$. Combining this with \eqref{1} and \eqref{2},
it follows that
\[
\overline{f}^{n}([F_{1}]_{\xi})\in \overline{f}^{n}({U})\cap {V}\neq \emptyset.
\]
\end{proof}

Being the end of this section, we shall prove that there exists $g\in D_{m}(I)$
such that the $g$-fuzzification system of every nontrivial dynamical system
is not transitive, giving a partial answer to Question \ref{Question 1}.
The following lemma is obvious.

\begin{lemma}\label{Lem-Transitivity}
A dynamical system $(X, f)$ is transitive if and only if for every pair of nonempty
open subsets $U, V$ of $X$, $N(U, V)\in \mathscr{F}_{inf}$.
\end{lemma}

\begin{theorem}\label{Wu-th}
Let $(X, d)$ be a nontrivial metric space and $g\in D_{m}(I)$ satisfying that there exist
$z\in (0, 1]$ and $m\in \mathbb{N}$ such that $\xi_{g}(z)\neq z$ and $\xi_{g}^{m+1}(z)
=\xi_{g}^{m}(z)$. Then,
for any $f\in \mathscr{C}(X)$, its $g$-fuzzification system $(\mathbb{F}^{1}(X), \widetilde{f}_{g})$
is not transitive, where $\mathscr{C}(X)$ is the set of all continuous self-maps defined on $X$.
\end{theorem}

\begin{proof}
Fix two distinct points $a, b\in X$, as $X$ is nontrivial. To prove
this theorem, consider two cases as follows:

\medskip

\textbf{Case 1.}\quad  $\xi_{g}(z)>z$. Since $\xi_{g}$ is nondecreasing, applying mathematical
induction, it follows that for any $j\in \mathbb{N}$,
\begin{equation}\label{E-1}
\xi_{g}^{j+1}(z)\geq \xi_{g}^{j}(z)\geq \xi_{g}(z)>z.
\end{equation}
Set
\[
E_{1}=\left\{x\in X: d(x, a)\leq \frac{d(a, b)}{8}\right\},
\]
and
\[
E_{2}=\left\{x\in X: d(x, b)\leq \frac{d(a, b)}{8}\right\}.
\]
Take two fuzzy sets $E, G\in \mathbb{F}^{1}(X)$ such that
\[
E(x)=\left\{\begin{array}{cc}
1, & {\rm } \ x\in E_{1}, \\\
z, & {\rm } \ x\in E_{2}, \\\
0, & {\rm } \ x\in X\setminus \left(E_{1}\cup E_{2}\right),
\end{array}
\right.
\]
and
\[
G(x)=\left\{\begin{array}{cc}
1, & {\rm } \ x\in E_{2}, \\\
z, & {\rm } \ x\in E_{1}, \\\
0, & {\rm } \ x\in X\setminus \left(E_{1}\cup E_{2}\right).
\end{array}
\right.
\]
Let
\[
\eta=\inf\left\{d(x, y): x\in E_{1}, \ y\in E_{2}\right\}\geq \frac{3}{4}d(a, b)
\]
and
\[
\mathcal{U}=\left\{F\in \mathbb{F}^{1}(X): d_{\infty}(F, E)<\frac{\eta}{4}\right\},
\]
\[
\mathcal{V}=\left\{F\in \mathbb{F}^{1}(X): d_{\infty}(F, G)<\frac{\eta}{4}\right\}.
\]
Since $d_{\infty}(E, G)\geq d_{H}([E]_{1}, [G]_{1})=d_{H}(E_{1}, E_{2})\geq \eta$, then
$\mathcal{U}\cap \mathcal{V}=\emptyset$.

Now, we claim that for any $n>m$, $(\widetilde{f}_{g})^{n}(\mathcal{U})\cap
\mathcal{V}=\emptyset$.

In fact, if there exist some $n> m$ such that $(\widetilde{f}_{g})^{n}(\mathcal{U})\cap
\mathcal{V}\neq\emptyset$, then there exists $P\in \mathcal{U}$ such that $(\widetilde{f}_{g})^{n}
(P)\in \mathcal{V}$. This implies that for any $\alpha\in (0, 1]$,
\[
d_{H}([(\widetilde{f}_{g})^{n}(P)]_{\alpha}, [G]_{\alpha})<\frac{\eta}{4}.
\]
In particular, applying \eqref{E-1} and Proposition \ref{Proposition 3.1}, it follows that
\[
d_{H}([(\widetilde{f}_{g})^{n}(P)]_{z}, [G]_{z})=d_{H}(\overline{f}^{n}([P]_{\xi_{g}^{n}(z)}),
E_{1} \cup E_{2})
<\frac{\eta}{4},
\]
and
\begin{eqnarray*}
d_{H}([(\widetilde{f}_{g})^{n}(P)]_{\xi_{g}(z)}, [G]_{\xi_{g}(z)})
&=&d_{H}(\overline{f}^{n}([P]_{\xi_{g}^{n+1}(z)}),
E_{2})\\
&=&d_{H}(\overline{f}^{n}([P]_{\xi_{g}^{n}(z)}),
E_{2})<\frac{\eta}{4}.
\end{eqnarray*}
So,
\begin{eqnarray*}
\frac{\eta}{2}& > & d_{H}(\overline{f}^{n}([P]_{\xi_{g}^{n}(z)}), E_{1} \cup E_{2})+ d_{H}(\overline{f}^{n}([P]_{\xi_{g}^{n}(z)}), E_{2})\\
& \geq & d_{H}(E_{1} \cup E_{2}, E_{2})=\max_{x\in E_{1} \cup E_{2}}\inf_{y\in E_{2}}d(x, y)\\
& = & \max_{x\in E_{1}}\inf_{y\in E_{2}}d(x, y)\geq \eta,
\end{eqnarray*}
which is a contradiction as $\eta>0$.

\medskip

\textbf{Case 2.}\quad  $\xi_{g}(z)<z$. Similarly to the proof of Case 1, it can be verified that
there exist nonempty open subsets $\mathcal{U}, \mathcal{V}$ of $\mathbb{F}^{1}(X)$ such that for any $n>m$, $(\widetilde{f}_{g})^{n}(\mathcal{U})\cap
\mathcal{V}=\emptyset$.

Summing up Case 1 and Case 2, applying Lemma \ref{Lem-Transitivity}, it follows that
$(\mathbb{F}^{1}(X), \widetilde{f}_{g})$ is not transitive.
\end{proof}

\begin{remark}\label{R-1}
Choose $g: I\longrightarrow I$ as
\[
g(x)=\begin{cases}
0, & {\rm } \ x=0, \\\
1-\frac{1}{2^{n}}, & {\rm } \ x\in [1-\frac{1}{2^{n}}, 1-\frac{1}{2^{n+1}}), n\in \mathbb{N}, \\\
1, & {\rm } \ x=1.
\end{cases}
\]
It can be verified that $g\in D_{m}(I)$, and
\[
\xi_g(x)=\begin{cases}
0, & {\rm } \ x=0, \\\
1-\frac{1}{2^{n+1}}, & {\rm } \ x\in (1-\frac{1}{2^{n}}, 1-\frac{1}{2^{n+1}}], n\in \mathbb{N}, \\\
1, & {\rm } \ x=1.
\end{cases}
\]
Clearly, $z=1/ 4$ satisfies that $\xi_{g}(z)=1/2 \neq z$, and $\xi_{g}^{n}(z)=1/2$ for all $n\geq 2$.
This, together with Theorem \ref{Wu-th}, implies that the answer to Question \ref{Question 1} is negative.
\end{remark}

\section{Conclusions}

In this paper, we present a systematic study of the sensitivity of $g$-fuzzification.
Firstly, we prove that $(\mathcal{K}(X), \overline{f})$ is sensitively dependent if
$(\mathbb{F}_{0}(X), \widetilde{f}_{g})$ is so (see Theorem \ref{Theorem 3.1}).
This, together with Example \ref{e-1-1}, gives a negative
answer to Question~\ref{Question 1} posed in \cite{Kupka2014}.
Then, we reveal some characteristics
ensuring that $(\mathcal{K}(X), \overline{f})$ is sensitive,
$\mathscr{F}$-sensitive, or multi-sensitive (see Theorem \ref{Theorem 1}, Theorem \ref{F-sensitivity},
and Theorem \ref{multi-sensitivity}, respectively).
Moreover, we show that $(\mathcal{K}(X), \overline{f})$ is
weakly mixing provided that $(\mathbb{F}^{1}(X), \widetilde{f}_{g})$ is transitive.
Finally, we prove that there exists $g\in D_{m}(I)$ such that for any dynamical system
$(X, f)$, $(\mathbb{F}^{1}(X), \widetilde{f}_{g})$ is not transitive (thus, not weakly mixing).

\section*{Acknowledgments}
We thank the referees for their careful reading and
valuable suggestions which helped us to improve the quality of this paper.

Xinxing Wu was by scientific research starting project
of Southwest Petroleum University (No. 2015QHZ029),  the Scientific Research
Fund of the Sichuan Provincial Education Department (No. 14ZB0007), and NSFC (No. 11401495).

Guanrong Chen was supported by the Hong Kong Research Grants Council under
GRF Grant CityU 11208515.

\bibliographystyle{amsplain}

\end{document}